\documentclass[a4paper,10pt,twoside]{article}

\usepackage[applemac]{inputenc}
\usepackage[dvipsnames]{xcolor} 

\usepackage{tikz}  

\usepackage{amsthm, amssymb, amsfonts, amsmath}	
\usepackage{latexsym}
\usepackage{graphicx,caption,subcaption}
\usepackage[all]{ xy}

\usepackage{microtype}
\usepackage{hyperref}
\hypersetup{
	colorlinks=true,
	linkcolor=black,
citecolor=blue}

\theoremstyle{plain}
\newtheorem{theorem}{Theorem}[section]
\newtheorem{lemma}[theorem]{Lemma}
\newtheorem{proposition}[theorem]{Proposition}
\newtheorem{corollary}[theorem]{Corollary}

\theoremstyle{definition}

\newtheorem{remark}[theorem]{Remark}

\newcommand{\R}{\mathbb{R}}

\newcommand{\tr}{\mathrm{tr}}
\newcommand{\ad}{\mathrm{ad}}
\newcommand{\Ad}{\mathrm{Ad}}
\newcommand{\Aut}{\mathrm{Aut}}
\newcommand{\GL}{\mathrm{GL}}
\newcommand{\spn}{span}
\newcommand{\SL}{\mathrm{SL}}

 \numberwithin{equation}{section}
\author{F. Rubilar  {\small and} L. Schultz}
\date{}
\title{Adjoint orbits of $\mathfrak{sl}(2,\mathbb{R})$ and their geometry}
\begin{document}
\maketitle

\thispagestyle{empty} 
\begin{abstract}
	Let $\SL(n,\mathbb{R})$ be the special linear group and $\mathfrak{sl}(n,\mathbb{R})$ its Lie algebra.  We study geometric properties associated to the adjoint orbits in the simplest non-trivial case, namely, those of $\mathfrak{sl}(2,\mathbb{R})$. In particular, we show that  just three possibilities arise: either the adjoint orbit is a one--sheeted hyperboloid, or a two--sheeted hyperboloid, or else a cone. 
	In addition, we introduce a specific potential and study the corresponding gradient vector field and its dynamics when we restricted to the adjoint orbit.  We conclude by describing the symplectic structure on  these adjoint orbits coming from the well known Kirillov--Kostant--Souriau symplectic form on coadjoint orbits.
\end{abstract}
\section{Introduction}

This text is of expository nature. 
We carry out the exercise of explicitly describing adjoint orbits of  $\mathfrak{sl}(2,\mathbb{R})$
together with the equations defining them as real affine algebraic varieties, 
over which we also describe  symplectic structures. 

We then focus on a single orbit that has the shape of a one-sheeted hyperboloid, 
presenting  it as a doubly ruled surface 
whose tangent bundle we describe explicitly.
We add a potential carefully  chosen to be a Morse
 function, and study the orbits of the corresponding  gradient flow. 
For the case of compact manifolds the classical Morse--Smale theorem  
states that the trajectories of the gradient flow 
converge to critical points of the potential. 
Here, in contrast,  we show that some trajectories
are not complete, thus highlighting the importance of the 
hypothesis of compactness in the Morse--Smale theory. 
For applications to mathematical physics it is essential to consider examples where some trajectories are not complete in time.

Even though our calculations are straightforward, 
we believe it is useful to have the results readily available in the literature. 
The study of the geometry of adjoint orbits is a classical topic in Geometry and Lie theory. 
However, the literature is mainly presented following an abstract approach, 
so, in this paper, we exhibit most of the details. 
Some references that focus in specific cases of adjoint and coadjoint orbits are   
\cite{BeHo}, for classical compact Lie groups, and \cite{Arv}, 
where there is an excellent explanation of the geometry of flag manifolds arising from the 
adjoint representation of compact semisimple Lie groups.\\

We study the geometry of those 
adjoint orbits which arise from the adjoint representation 
$\Ad\colon \SL(2, \R)\rightarrow\mathfrak{gl}(\mathfrak{sl}(2,\R))$, 
where for each $g\in\SL(2, \R)$ and $H\in\mathfrak{sl}(2,\R)$ the adjoint action is  $\Ad_g(H)=gHg^{-1}$.
Let $A, B, C$ be the basis of $\mathfrak{sl}(2,\mathbb{R})$ given by
\begin{equation}
A=\left[\begin{array}{cc}
0&1\\
1&0
\end{array}\right],\qquad B=\left[\begin{array}{cc}
1&0\\
0&-1
\end{array}\right],\qquad C=\left[\begin{array}{cc}
0&1\\
-1&0
\end{array}\right].
\label{baseSL2}
\end{equation}
We 
decompose $H=xA+yB+zC$ and then find out that the adjoint orbits are of one of the following three types:
\begin{itemize}
	\item a \textbf{one--sheeted hyperboloid}, given by the equation
	\[\mathcal{O}\colon x^2+y^2-z^2=\lambda^2,\hspace{1cm} \lambda\neq0;\]
	\item a \textbf{two--sheeted hyperboloid}, given by
	\[\mathcal{O}_{1}^{+}\colon x^2+y^2-z^2=-\lambda^2,\hspace{1cm} z>0,\lambda\neq0,\]
	\[\mathcal{O}_{1}^{-}\colon x^2+y^2-z^2=-\lambda^2,\hspace{1cm} z<0,\lambda\neq0;\]
	\item a \textbf{cone}, given by
	\[\mathcal{O}_{2}^{+}\colon x^2+y^2-z^2=0,\hspace{1cm} z>0,\]
	\[\mathcal{O}_{2}^{-}\colon x^2+y^2-z^2=0,\hspace{1cm} z<0,\]
	\[\mathcal{O}_{2}^{0}=\{0\}.\]
\end{itemize}

\vspace{3mm}

We  endow the adjoint orbit  $\mathcal{O}$  with the
symplectic structure arising from a coadjoint  orbit, thus  realizing it 
 as a  symplectic manifold.  
Namely, we use the isomorphism between adjoint and coadjoint orbits provided by the 
Killing form to give   this adjoint orbit 
  the symplectic structure pulled-back  from the  well known 
Kirillov--Kostant--Souriau form on the corresponding coadjoint orbit.

We then consider the function $f(x, y , z)=yz$  over $\mathfrak{sl}(2, \R)$ 
and regard its restriction to the orbit  $\mathcal{O}$
as a Morse function, calculating 
the trajectories of its  gradient vector field.
We  analyse the limit points of the gradient flow, 
and  compare the results obtained here to well known results about Morse flows for the compact case.

We observe that every orbit of the adjoint action on $\mathfrak{sl}(2, \R)$ 
is of one of the three types presented here, hence we 
have a complete description.

In general, understanding details of the family of all adjoint orbits 
for a given Lie algebra is a deep question with applications to non trivial aspects of the theory.
Some such research areas, among many,  are: the theory of Slodowy slices, 
the Springer theory, and  the Fukaya categories in homological mirror aymmetry.
Therefore, the calculations we present here may be regarded as a warm up  
exercise in preparation to the study of more advanced topics.

\section{Preliminaries}

We start by recalling some basic definitions of Lie theory. For further details, we suggest \cite{ErWi, SM}.

A \textbf{Lie group} is a smooth manifold $G$ with a smooth map from 
$G\times G\rightarrow G$ that makes $G$ into a group and such that the inverse map 
$g\mapsto g^{-1}$ is also smooth.

Let $\mathrm{M}(n,\mathbb{R})$ be the set of $n\times n$ matrices with entries in the real numbers.

The {\bf general linear group $\mathrm{GL}(n,\mathbb{R})$} is the subset of 
$\mathrm{M}(n,\mathbb{R})$ of non-singular matrices with matrix multiplication as group operation. 

By definition a {\bf matrix Lie group} is a  closed subgroup of $\GL(n,\R)$.

	For example, the  {\bf special linear group $\SL(n,\mathbb{R})$} 
is the subgroup of $\mathrm{GL}(n,\mathbb{R})$ of non-singular matrices of determinant $1$.

	A \textbf{Lie algebra} is a vector space $\mathfrak{g}$ over a field $\mathbb{F}$ 
together with a \textbf{Lie bracket}, that is, a bilinear map 
	$$\mathfrak{g}\times \mathfrak{g} \rightarrow \mathfrak{g},\qquad (x, y)\mapsto \left[x,y\right],$$
	satisfying
	\begin{itemize}
		\item $\left[x,x\right]=0 \textup{ for each } x\in \mathfrak{g},$ 
		\item Jacobi identity: $\left[x,\left[y,z\right]\right]+\left[y,\left[z,x\right]\right]+\left[z,\left[x,y\right]\right]=0\textup{ for every }  x,y,z\in \mathfrak{g}.$
	\end{itemize}
\begin{remark}
	If the characteristic of the $\mathbb{F}$ is not $2$, then the first condition is 
equivalent to anticommutativity \[\left[x, y\right]=-\left[y, x\right]\textup{ for each }  x,y \in \mathfrak{g}.\]
\end{remark}

The \textbf{centre} of a Lie algebra  consists of all those elements 
$x$ in $\mathfrak{g}$, subject to $[x,y] = 0$ for all $y$ in $\mathfrak{g}$.

Let $\mathfrak{g}_1, \mathfrak{g}_2$ be two Lie algebras over a field $\mathbb{F}$. 
A map $\varphi\colon\mathfrak{g}_1\rightarrow\mathfrak{g}_2$ is a \textbf{Lie algebra homomorphism} 
if  $\varphi$ is linear and satisfies $$\varphi(\left[x, y\right])=\left[\varphi(x), \varphi(y)\right],$$ 
for each $x,y\in \mathfrak{g}_1$. If $\varphi$ is  bijective, we call it an \textbf{isomorphism}.

There are several ways to understand the Lie algebra of a Lie group. 
Here we consider it as the tangent space at the identity element of the group, 
that is, if $G$ is a Lie group, then its Lie algebra $\mathfrak{g}$ corresponds to $T_eG$.

For instance, $\SL(n, \R)$ is a matrix Lie group with Lie algebra $\mathfrak{sl}(n,\R)$. 
In terms of matrices, $\mathfrak{sl}(n,\mathbb{R})$ is the Lie algebra of $n\times n$ 
matrices  with trace $0$ and coefficients in $\mathbb{R}$, 
where the Lie bracket is the usual commutator $[X,Y]=XY-YX$.

Let $A$ be a $n\times n$ matrix over $\R$ or $\mathbb{C}$. The \textbf{exponential} of $A$ is the $n\times n$ matrix
\[\exp(A)=\sum_{k=0}^{\infty}\frac{A^k}{k!}.\]

	An important result in Lie theory is that if $G$ is a matrix Lie group 
with algebra $\mathfrak{g}$, then $\exp(A)\in G$ holds for each $A\in\mathfrak{g}$. 
Below we provide a direct proof when $G=\SL(n,\R)$.

\begin{proposition}
	For any $A\in\mathfrak{sl}(n, \R)$, we have $\exp(A)\in\SL(n, \R)$.
\end{proposition}
\begin{proof}
	Consider the Jordan form of $A$. If $\{\lambda_i\}_{i=1}^{l}$ are the eigenvalues of $A$, then we have
	\begin{align*}
		\exp(A)=\sum_{k=0}^{\infty}\frac{A^k}{k!}=&\left[\begin{array}{cccc}
			\sum_{k=0}^{\infty}\frac{\lambda_1^k}{k!}&*&*&*\\
			0&\sum_{k=0}^{\infty}\frac{\lambda_2^k}{k!}&*&*\\
			\vdots&\vdots&\ddots&\vdots\\
			0&0&0&\sum_{k=0}^{\infty}\frac{\lambda_l^k}{k!}
		\end{array}\right]\\
		=&\left[\begin{array}{cccc}
			e^{\lambda_1}&*&*&*\\
			0&e^{\lambda_2}&*&*\\
			\vdots&\vdots&\ddots&\vdots\\
			0&0&0&e^{\lambda_l}
		\end{array}\right].
	\end{align*}
We get immediately the equality
	\[\det(\exp(A))=\prod_{i=1}^{l}e^{\lambda_i}=e^{\sum_{i=1}^{l}\lambda_i}=e^{\tr(A)}=e^0=1,\]
	and so $\exp(A)\in\SL(n, \R)$.
\end{proof}

	Let $G$ be a Lie group with Lie algebra $\mathfrak{g}$.	
The \textbf{adjoint representation of G}  on $\mathfrak{g}$ is the group homomorphism
	\[\begin{array}{rcl}
	\Ad\colon G&\rightarrow &\Aut(\mathfrak{g})\\
	g&\mapsto&\Ad_g.
	\end{array}\]

For example, for $G=\mathrm{SL}(n,\R)$ and $\mathfrak{g}=\mathfrak{sl}(n,\R)$, the group homomorphism
is given by
  \[\begin{array}{rcl}
 \Ad\colon\SL(n,R)&\rightarrow &\Aut(\mathfrak{sl}(n,\mathbb{R}))\\
 g&\mapsto&\Ad_g,
 \end{array}\]
where $\Ad_g(X)=gXg^{-1}$ for every $X\in\mathfrak{sl}(n,\mathbb{R})$.

Given  $H\in\mathfrak{sl}(n,\mathbb{R})$, its \textbf{adjoint orbit} is \[\mathcal{O}(H)=\{gHg^{-1}:g\in\SL(n,\mathbb{R})\}.\]


We will see that the geometric structure on adjoint orbits depends strongly on the element $H\in\mathfrak{sl}(2,\mathbb{R})$. 
We will give a complete characterization of those orbits.

	The \textbf{adjoint representation of the Lie algebra} $\mathfrak{g}$ 
in $\mathfrak{gl}(\mathfrak{g})$ is the homomorphism
	\[\begin{array}{rcl}
	\ad\colon\mathfrak{g}&\rightarrow &\mathfrak{gl}(\mathfrak{g})\\
	x&\mapsto&\ad_x,
	\end{array}\]
here $\ad_x(y)=\left[x,y\right]$ for each $x,y\in\mathfrak{g}$.

It follows by bilinearity of the Lie bracket that $\ad_x$ is linear for each $x\in\mathfrak{g}$; 
the same  is true for the correspondence $x\mapsto \ad_x$. 
In order to prove that $\ad_x$ is a homomorphism we just  have to check that  $\ad_x$ satisfies the identity
\[\ad(\left[x,y\right])=\ad_x\circ\ad_y-\ad_y\circ\ad_x,\qquad \textup{ for every }x,y\in\mathfrak{g}.\]
The above equality holds precisely because of the Jacobi identity. 
The kernel of $\ad$ is the \textbf{centre} of $\mathfrak{g}$.

\section{Adjoint orbits of $\mathfrak{sl}(2,\R)$}
Here we study the geometry of orbits of $\mathfrak{sl}(2,\R)$ given by the adjoint action, 
namely, the action induced by the adjoint representation of $\SL(2,\R)$ in its  associated Lie algebra $\mathfrak{sl}(2, \R)$. 
We will classify them into three classes: either the adjoint orbit is a one--sheeted hyperboloid, 
or a two--sheeted hyperboloid, or else a cone, 
depending on the choice of the element that we take in the Lie algebra $\mathfrak{sl}(2,\R)$. 
Recall the basis of $\mathfrak{sl}(2,\R)$ introduced in (\ref{baseSL2}), namely
$$A=\left[\begin{array}{cc}
	0&1\\
	1&0
\end{array}\right],\qquad B=\left[\begin{array}{cr}
	1&0\\
	0&-1
\end{array}\right],\qquad C=\left[\begin{array}{rr}
	0&1\\
	-1&0
\end{array}\right].$$
\subsection{The one--sheeted hyperboloid}
Here we study the orbit of $\lambda A$ in $\mathfrak{sl}(2,\R)$ for $\lambda\in\R\backslash\{0\}$.\\

Let be $H$ in $\mathfrak{sl}(2,\R)$ and consider the decomposition
\[H=xA+yB+zC;\quad x, y, z\in\R.\]

\begin{proposition}
For fixed $\lambda\neq0$, the adjoint orbit $\mathcal{O}(\lambda A)$ 
is the set of matrices $H=xA+yB+zC$ in $\mathfrak{sl}(2,\R)$ that satisfy 	
$$x^2+y^2-z^2=\lambda^2.$$ 
\end{proposition}
\begin{proof}
First we prove that if $H$ belongs to such orbit, then $x^2+y^2-z^2=\lambda^2.$
The adjoint orbit of $\lambda A$  is by definition
\[\mathcal{O}(\lambda A)=\{g\lambda Ag^{-1}:g\in\SL(2,\mathbb{R})\}.\]
Hence,  if $H\in\mathcal{O}(\lambda A)$, there exists  $M\in\SL(2,\R)$ such that $H=M\lambda AM^{-1}.$ 
Since the determinant of a matrix is invariant under conjugation, we have
\[\det(H)=\det(\lambda A).\]
Thus, we obtain
\begin{align*}
\det(H)&=\det\left(\left[\begin{array}{cc}
y&x+z\\
x-z&-y	
\end{array}\right]\right)=z^2-x^2-y^2,\\
\det(\lambda A)&=\det\left(\left[\begin{array}{cc}
0&\lambda\\
\lambda&0	
\end{array}\right]\right)=-\lambda^2,
\end{align*}
which implies 
\begin{equation}
x^2+y^2-z^2=\lambda^2. 
\label{osh}
\end{equation}
Thus we conclude the first part of the proof. It is a well known fact 
that Equation (\ref{osh}) defines a surface in $\R^3$ called a \textbf{one--sheeted hyperboloid}.\\

Now we show that, reciprocally, if $H=xA+yB+zC$ satisfies Equation (\ref{osh}), 
then $H$ belongs to $\mathcal{O}(\lambda A)$. 
Given any matrix $N\in \mathfrak{sl}(2, \mathbb{R})$, 
its characteristic polynomial is completely determined by its determinant. 
Indeed, if $\rho_N$ denotes the characteristic polynomial of $N$, we have
$$
\rho_N(t)=t^2+\mathrm{det}(N).
$$
Thus, once $H$ satisfies Equation (\ref{osh}), we get $\mathrm{det}(H)=-\lambda^2$ and therefore
$$
\rho_H(t)=t^2-\lambda^2=(t-\lambda)(t+\lambda).
$$
As soon as $\lambda$ is assumed to be different than zero, 
we know that $H$ has two distinct eigenvalues, and so $H$ is diagonalizable. Let
$$
D=\left[\begin{array}{cc}
\lambda&0\\
0&-\lambda	
\end{array}\right]
$$
and $P\in\mathrm{GL}(2,\mathbb{R})$ be such that $PHP^{-1}=D$. 
Note that we can assume $P\in \mathrm{SL}(2,\mathbb{R})$ by 
multiplying its first column by $\frac{1}{\det(P)}$ if necessary. 
By the same argument, we find $P_0\in \mathrm{SL}(2,\mathbb{R})$ 
such that $P_0\lambda AP_{0}^{-1}=D$. Thus, we get
$$
(P_0^{-1}P)H(P_0^{-1}P)^{-1}=\lambda A
$$
with $P_0^{-1}P\in \mathrm{SL}(2,\mathbb{R})$. 
We conclude that if $H=xA+yB+zC$ satisfies Equation (\ref{osh}), 
then $H$ belongs to the orbit $\mathcal{O}(\lambda A)$ and we are done.
\end{proof}

\begin{remark}
	Since $\det(\lambda B)$ satisfies Equation (\ref{osh}), 
the above argument implies $\mathcal{O}(\lambda A)=\mathcal{O}(\lambda B)$.
\end{remark}
\begin{remark}
	In the complex case, i.e., for $\mathfrak{sl}(2,\mathbb{C})$, if we consider $$H_0=\left[\begin{array}{cc}
	1&0\\
	0&-1
	\end{array}\right],$$
	then we get that its adjoint orbit $\mathcal{O}(H_0)$ is diffeomorphic to 
$T^*\mathbb{P}^1$, specifically, the cotangent bundle of the complex projective line. 
So, the geometric structure of the adjoint orbit is quite different. 
Moreover, in \cite{GGSM1}, Gasparim, Grama, and San Martin gave a complete description 
of the diffeomorphism type of adjoint orbits for diagonal matrices in $\mathfrak{sl}(n,\mathbb{C})$.
\end{remark}
\subsection{The two--sheeted hyperboloid} Now we turn to the geometric 
structure of the adjoint orbit of $\lambda C$.
\begin{proposition}
Fix $\lambda \in\mathbb{R}\backslash\{0\}$. The adjoint orbit $\mathcal{O}(\lambda C)$ is the set of matrices $H=xA+yB+zC$ in $\mathfrak{sl}(2,\R)$ subject to
$$x^2+y^2-z^2=-\lambda^2.$$
\end{proposition}
\begin{proof}
For $H\in\mathcal{O}(\lambda C)$ there exists $N\in\SL(2,\R)$ such that $N\lambda CN^{-1}=H.$ 
Therefore we have
\[\det(\lambda C)=\det(NHN^{-1})=\det(H),\]
and so we get
\begin{align*}
\det(H)&=\det\left(\left[\begin{array}{cc}
y&x+z\\
x-z&-y	
\end{array}\right]\right)=z^2-x^2-y^2,\\
\det(\lambda C)&=\det\left(\left[\begin{array}{cc}
0&\lambda\\
-\lambda&0	
\end{array}\right]\right)=\lambda^2,
\end{align*}
which implies 
\begin{equation}
x^2+y^2-z^2=-\lambda^2 .
\label{tsh}
\end{equation}

For the reciprocal, we start by showing that there is no $M\in \mathrm{SL}(n, \mathbb{R})$ 
such that $M(\lambda C) M^{-1}=-\lambda C$. Without loss of generality take $\lambda>0$. Then, for
$$
M=\left[\begin{array}{cc}
u&v\\
s&t	
\end{array}\right]
$$
we reach
$$
M(\lambda C)M^{-1}=\lambda\left[\begin{array}{cc}
-us-tv&u^2+v^2\\
-s^2-t^2&us+tv	
\end{array}\right].
$$
As $u^2+v^2\geq 0$, we easily conclude that there is no $M\in \mathrm{SL}(n, \mathbb{R})$ 
such that $M(\lambda C) M^{-1}=-\lambda C$. 
The bottom line is that we have $\lambda C\in\mathcal{O}_1^+$ if $\lambda>0$ and 
$\lambda C\in\mathcal{O}_1^-$ if $\lambda<0.$ 

Next we show that if $H=xA+yB+zC$ is such that $x, y, z$ satisfy (\ref{tsh}), 
then $H$ belongs to $\mathcal{O}_1^+$ or $\mathcal{O}_1^-$. 
To verify this, we use an argument similar to the one we used in the previous subsection. 
Once $H$ is such that Equation (\ref{tsh}) holds, its characteristic polynomial is given by
$$
\rho_H(t)=t^2+\lambda^2=(t+i\lambda)(t-i\lambda).
$$
So, we can write $H$ in its real Jordan form in either of two different ways
$$
PHP^{-1}=\left[\begin{array}{cc}
0&\lambda\\
-\lambda&0	
\end{array}\right]
$$
or
$$
RHR^{-1}=\left[\begin{array}{cc}
0&-\lambda\\
\lambda&0	
\end{array}\right],
$$
always with $R, P \in \mathrm{GL}(2,\mathbb{R})$. 
The structural  difference between these two cases is that if $\det(R)>0$ 
then $\det(P)<0$, and vice versa. Assume $\det(P)>0$. 
Then we can define $\tilde{P}=\frac{1}{\sqrt{\det(P)}}P$ in order to get
$$
\tilde{P}H\tilde{P}^{-1}=\lambda C,
$$
where $\tilde{P}\in \mathrm{SL}(2,\mathbb{R})$ 
and we conclude that $H$ belongs to the orbit $\mathcal{O}(\lambda C)$. In the case when $\det(P)<0$, we repeat the same construction for $R$ in order to get
$$
\tilde{R}H\tilde{R}^{-1}=\lambda C.
$$

\end{proof}
\begin{remark}
Equation (\ref{tsh}) defines a  \textbf{two--sheeted hyperboloid}.
In this case, we have two situations. 
Either $z>0$ (the upper half part of the hyperboloid) or $z<0$ (the lower half), which respectivelly correspond to
\[\mathcal{O}_{1}^{+}\colon x^2+y^2-z^2=-\lambda^2,\hspace{1cm} z>0,\lambda\neq0,\]
\[\mathcal{O}_{1}^{-}\colon x^2+y^2-z^2=-\lambda^2,\hspace{1cm} z<0,\lambda\neq0.\]
\end{remark}


\subsection{The cone}Note that we have analysed adjoint orbits of matrices with determinant either positive or negative. 
In this section we study the remaining situation, namely, adjoint orbits of matrices with zero determinant. 
In order to do this, define $D=A+C$ and consider its adjoint orbit $\mathcal{O}(\lambda D)$. 
The main result reads as follows.
\begin{proposition}
The adjoint orbit $\mathcal{O}(\lambda D)$ corresponds to matrices 
$H=xA+yB+zC\in\mathfrak{sl}(2,\R)$ subject to $x^2+y^2-z^2=0.$
\end{proposition}
\begin{proof}
If $H=xA+yB+zC$ belongs to $\mathcal{O}(\lambda D)$, 
we can write down $H=L\lambda DL^{-1}$ where $L\in\SL(2,\R)$. As before we get
\begin{align*}
\det(H)&=\det\left(\left[\begin{array}{cc}
y&x+z\\
x-z&-y	
\end{array}\right]\right)=z^2-x^2-y^2\, \textup{ and}\\
\det(\lambda D)&=\det\left(\left[\begin{array}{cc}
0&2\lambda\\
0&0
\end{array}\right]\right)=0.
\end{align*}
We conclude that if $H=xA+yB+zC$ belongs to $\mathcal{O}(\lambda D)$, then $x, y, z$ satisfy the relation
\begin{equation}
x^2+y^2-z^2=0.
\label{cone}
\end{equation}

Next we show that if $H=xA+yB+zC$ is such that $x, y, z$ satisfy (\ref{cone}), 
then $H$ belongs either to $\mathcal{O}_2^+$, $\mathcal{O}_2^-$ or $\mathcal{O}_2^0$. 
To prove this, we look again to the Jordan form of $H$. 
Now, once we know $\rho_H(t)=t^2$, we have two cases: whether $H=0$ or $H$ has as Jordan form
$$
\left[\begin{array}{cc}
0 & 1\\
0 &0	
\end{array}\right].
$$

If $H=0$, we have $H\in \mathcal{O}^{0}_{2}$. 
So assume $H\neq 0$ and let $P\in \mathrm{GL}(2,\mathbb{R})$ be such that
$$
PHP^{-1}=\left[\begin{array}{cc}
0 & 1\\
0 &0	
\end{array}\right].
$$
If $\det(P)>0$, define $\tilde{P}=\frac{1}{\sqrt{\det(P)}}P \in \mathrm{SL}(2, \mathbb{R})$ 
in order to obtain $H\in \mathcal{O}_2^+$.
On the other hand, if $\det(P)<0$, we can define $\tilde{P} \in \mathrm{SL}(2, \mathbb{R})$ 
as the matrix that we obtain from $P$ by multiplying its first column by $\displaystyle \frac{1}{\det(P)}$. 
And so, we reach
$$
\tilde{P}H\tilde{P}^{-1}=\left[\begin{array}{cc}
0 & \det(P)\\
0 &0	
\end{array}\right].
$$
Therefore, in this case we obtain $H\in \mathcal{O}_2^-$.
Note that if $\lambda>0$, then we get always $x+z>0$ and $x-z<0$, hence we have $z>0$ and therefore
\[\mathcal{O}_2^+=\{H=xA+yB+zC\in\mathfrak{sl}(2,\R)\colon x^2+y^2-z^2=0, z>0\}.\]
When $\lambda<0$, we have $x+z<0$ and $x-z>0$ so we get $z<0$ and thus
\[\mathcal{O}_2^-=\{H=xA+yB+zC\in\mathfrak{sl}(2,\R)\colon x^2+y^2-z^2=0, z<0\}.\]
Finally, for $\lambda=0$ we have 
\[\mathcal{O}_2^0=\{0\}.\]

Equation (\ref{cone}) defines a  \textbf{cone}.
We distinguish three situations; either $z>0$ (upper half of the cone), 
or $z<0$ (lower half), or else $z=0$ (origin); 
which are determined by three different orbits denoted by 
$\mathcal{O}_2^+$, $\mathcal{O}_2^-$ and $\mathcal{O}_2^0$, respectively. 
We claim that we have
\begin{itemize}
	\item $\lambda D\in\mathcal{O}_2^+$ if $\lambda>0$,
	\item $\lambda D\in\mathcal{O}_2^-$ if $\lambda<0$, and
	\item $\lambda D\in\mathcal{O}_2^0$ if $\lambda=0.$
\end{itemize} 
To see this, take $M\in \mathrm{SL}(n, \mathbb{R})$ and write
$$
M=\left[\begin{array}{cc}
u&v\\
s&t	
\end{array}\right].
$$
We have then 
$$
M(\lambda D)M^{-1}=\lambda\left[\begin{array}{cc}
-2su & 2u^2\\
-2s^2 &2us	
\end{array}\right].
$$
So, if $\lambda> 0$, we have $2\lambda u^2 \geq 0$; while if $u=0$, necessarily $s\neq 0$, 
otherwise must have $\det(M)=0$. 
So, we conclude that there are three exclusive orbits associated with Equation (\ref{cone}), 
depending on whether $\lambda$ is positive, negative  or zero.
\end{proof}

\begin{remark}
	Now it is trivial to see that these objects comprise the adjoint orbits of $\mathfrak{sl}(2,\R)$. 
In fact, given a non-zero matrix $H\in\mathfrak{sl}(2,\R)$ subject to $H=xA+yB+zC$, 
we have $x^2+y^2-z^2=\alpha\in\R.$ 
In this way, if $\alpha\in\R^+$ then $H\in\mathcal{O}(\lambda A)$. If $\alpha\in\R^-$ 
then $H\in\mathcal{O}_1^+$ or $H\in\mathcal{O}_1^-$, while for $\alpha=0$ 
we have $H\in\mathcal{O}_2^+$ or $H\in\mathcal{O}_2^-$. 
If $H$ is the zero matrix, of course we get $H\in\mathcal{O}_2^0$. 
Thus, every element in $\mathfrak{sl}(2,\R)$ is contained in one and only one of these orbits.
	\end{remark}
\begin{figure}[h]
	\centering      
	\includegraphics[width=0.8\textwidth]{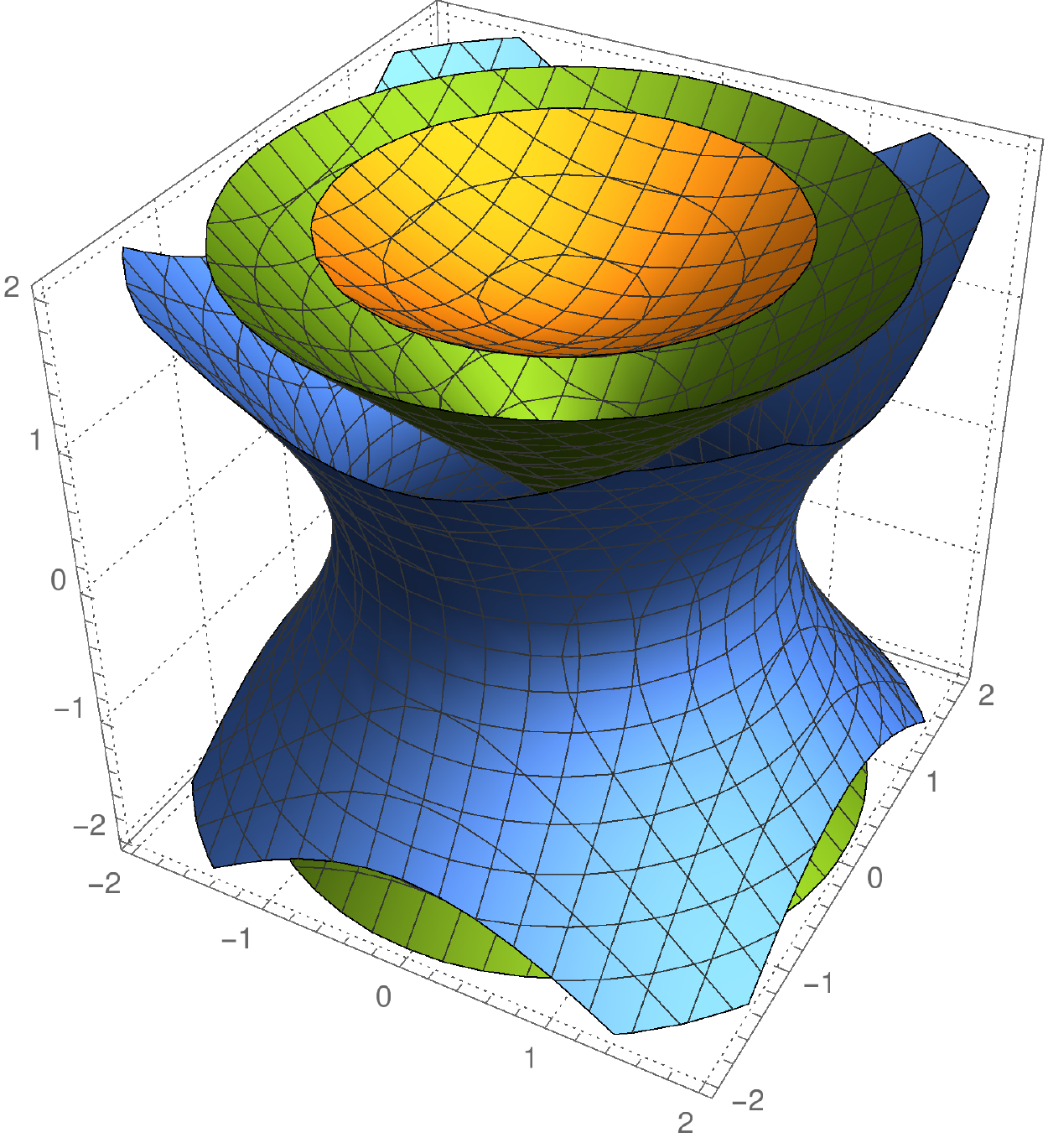}      
	\caption{Orbits of $\mathfrak{sl}(2,\R)$.}
	\label{fig:orbits}
\end{figure}

\section{The geometry of the one--sheeted hyperboloid}\label{section3.1}
Here we show that the one--sheeted hyperboloid is a ruled surface. 
Next we use this result to study the dynamics of a gradient field restricted to this surface.

Recall that a surface $S$ is called \textbf{ruled} if it is the union of a one 
parameter family of lines $\{r_{\alpha}\}_{\alpha \in \mathcal{A}}$. More precisely, there is a family of lines $\{r_{\alpha}\}_{\alpha \in \mathcal{A}}$ and a parametrization $r$ of $S$ satisfying the following properties. 

\begin{itemize}
    \item The parametrization $r$ is of the form $r(u, v)= c(u)+vb(u)$,  for a given $v \in \mathbb{R}$, where $c$ and $b$ are smooth functions.
	\item For each $u$, there is $\alpha_u\in \mathcal{A}$, such that $r_{u}(v)=r(u, v)$ is the parametric equation of the line $r_{\alpha_u}\in \{r_\alpha\}_{\alpha\in \mathcal{A}}$.
	
	\item The association $u\mapsto \alpha_u$ is a one to one correspondence.
\end{itemize} 
In this case, we say that $S$ is ruled by the family $\{r_\alpha\}_{\alpha \in \mathcal{A}}$.

Similarly, $S$ is called \textbf{doubly ruled} if it can be ruled in different ways by two disjoint families of lines.

For now on, denote by $S$ the one--sheeted hyperboloid given by the equation 
$x^2+y^2-z^2=\lambda^2$, with fixed $\lambda\neq0$. 
In order to show that $S$ is doubly ruled, we will construct explicitly such families.

\begin{lemma}
	Let $S$ be the one--sheeted hyperboloid given by equation $x^2+y^2-z^2=\lambda^2$, with $\lambda\neq0
	$. There exist two disjoint families of lines $F_1, F_2$ contained in $S$.
	\label{lemma2}
\end{lemma}
\begin{proof}
Consider the cylinder $C\colon x^2+y^2=\lambda^2$ which intersect $S$ in the plane $z=0$. 
Let $(x_0, y_0)$ be a point in the circle $C\cap S$. Think of
$y=f(x)=\pm\sqrt{\lambda^2-x^2}$ as describing the cylinder. 
Notice that the tangent space to $C$ at $(x_0, y_0)$ is given by
\[y-y_0=\frac{\partial f}{\partial x}(x_0)(x-x_0)+\frac{\partial f}{\partial z}(x_0)(z-0).\]
In particular for $y_0>0$, we get
\begin{equation}
y+\left(\frac{x_0}{\sqrt{\lambda^2-x_0^2}}\right)x=\frac{\lambda^2}{\sqrt{\lambda^2-x_0^2}},
\label{eq:4}
\end{equation}
while, for $y_0<0$ the equation is
\begin{equation}
y-\left(\frac{x_0}{\sqrt{\lambda^2-x_0^2}}\right)x=\frac{-\lambda^2}{\sqrt{\lambda^2-x_0^2}}.
\label{eq:5}
\end{equation}
Let us analyze each case separately.

\textbf{Case $y_0>0$.} We describe the intersection of the tangent space with $S$. Rewriting Equation (\ref{eq:4}) as

\begin{equation*}
y=\frac{\lambda^2}{\sqrt{\lambda^2-x_0^2}}-\left(\frac{x_0}{\sqrt{\lambda^2-x_0^2}}\right)x,
\end{equation*}
squaring both sides, and substituting $y^2$ by $\lambda^2+z^2-x^2$, we find
$$
(x-x_0)^2=\left(\frac{\lambda^2-x_0^2}{\lambda^2}\right)z^2.
$$
The above equation gives two planes containing $(x_0, y_0, 0)$, namely
	\begin{equation*}
	\left\{ \begin{array}{l}
	 x-\left(\frac{\sqrt{\lambda^2-x_0^2}}{\lambda}\right)z= x_0\\
	\\ 	x+\left(\frac{\sqrt{\lambda^2-x_0^2}}{\lambda}\right)z= x_0.
	\end{array}
	\right.
	\end{equation*}
Once again, considering the intersection with the plane (\ref{eq:4}), we get two sets of systems of equations
	\begin{equation}
	\left\{ \begin{array}{l}
	y-\left(\frac{x_0}{\sqrt{\lambda^2-x_0^2}}\right)x=\frac{-\lambda^2}{\sqrt{\lambda^2-x_0^2}}\\
	\\ x-\left(\frac{\sqrt{\lambda^2-x_0^2}}{\lambda}\right)z= x_0,
	\end{array}
	\right.
	\label{eq:6}
	\end{equation}
	\begin{equation}
	\left\{ \begin{array}{l}
	y-\left(\frac{x_0}{\sqrt{\lambda^2-x_0^2}}\right)x=\frac{-\lambda^2}{\sqrt{\lambda^2-x_0^2}}\\
	\\ x+\left(\frac{\sqrt{\lambda^2-x_0^2}}{\lambda}\right)z= x_0.
	\end{array}
	\right.
	\label{eq:7}
	\end{equation}

	Let us find the line determined by the planes in Equation (\ref{eq:6}). 
Note that $v_1=\left(\frac{x_0}{\sqrt{\lambda^2-x_0^2}}, 1, 0 \right)$ and 
$v_2=\left(1, 0, \frac{-\sqrt{\lambda^2-x_0^2}}{\lambda}\right)$ are normal vectors to the planes. 
We need the explicit value
	\[v_1\times v_2 = \left(\frac{-\sqrt{\lambda^2-x_0^2}}{\lambda}, \frac{x_0}{\lambda}, -1\right).\]
In this way, the parametric equation of the intersection line determined by (\ref{eq:6}) is
	\[r_1(t)=(x_0, y_0, 0)+t\left( \displaystyle\frac{-\sqrt{\lambda^2-x_0^2}}{\lambda}, \frac{x_0}{\lambda}, -1 \right).\]
	We check now why $r_1$ is contained in $S$. 
Using $y_0=\sqrt{\lambda^2-x_0^2}$, we get
	\[\left(x_0-\frac{t\sqrt{\lambda^2-x_0^2}}{\lambda}\right)^2+	\left(y_0+\frac{tx_0}{\lambda}\right)^2-t^2=\lambda^2.\]
	
Similarly, fashion the line determined by the planes in Equation (\ref{eq:7}) is
	\[r_2(t)=(x_0, y_0, 0)+t\left(\frac{\sqrt{\lambda^2-x_0^2}}{\lambda}, \frac{-x_0}{\lambda}, -1\right).\]
 Since $y_0=\sqrt{\lambda^2-x_0^2}$, for every $t\in\mathbb{R}$, we obtain
	\[\left(x_0+\frac{t\sqrt{\lambda^2-x_0^2}}{\lambda}\right)^2+	\left(y_0+\frac{tx_0}{\lambda}\right)^2-t^2=\lambda^2,\]
	and hence $r_2$ is contained in $S$.

	\textbf{Case $y_0<0$.} Substituting Equation (\ref{eq:5}) in $x^2+y^2-z^2=\lambda^2$, we reach
	\begin{equation}
	\left\{ \begin{array}{l}
	y-\left(\frac{x_0}{\sqrt{\lambda^2-x_0^2}}\right)x=\frac{-\lambda^2}{\sqrt{\lambda^2-x_0^2}}\\
	\\ x-\left(\frac{\sqrt{\lambda^2-x_0^2}}{\lambda}\right)z= x_0,
	\end{array}
	\right.
	\label{eq:8}
	\end{equation}
	\begin{equation}
	\left\{ \begin{array}{l}
    y-\left(\frac{x_0}{\sqrt{\lambda^2-x_0^2}}\right)x=\frac{-\lambda^2}{\sqrt{\lambda^2-x_0^2}}\\
	\\ x+\left(\frac{\sqrt{\lambda^2-x_0^2}}{\lambda}\right)z= x_0.
	\end{array}
	\right.
	\label{eq:9}
	\end{equation}
	
	Working as in the previous case, the intersection plane is
	\[s_1(t)=(x_0, y_0, 0)+t\left( \displaystyle\frac{-\sqrt{\lambda^2-x_0^2}}{\lambda}, \frac{-x_0}{\lambda}, -1 \right),\]
	which shows that $s_1$ is contained in $S$. 
	
Equation (\ref{eq:9}) yields
	\[s_2(t)=(x_0, y_0, 0)+t\left(\frac{\sqrt{\lambda^2-x_0^2}}{\lambda}, \frac{x_0}{\lambda}, -1\right),\]
	also contained in $S$.

	For $(\lambda, 0, 0)$ and $(-\lambda, 0, 0)$, namely the point when $y_0=0$, 
the tangent spaces are given by the equations $x=\lambda$ and $x=-\lambda$, respectively. 
When $x=\lambda$ both
	\begin{equation}
	\begin{aligned}
	l_1(t)=&(\lambda, 0, 0)+t(0, 1, -1),\\
	l_2(t)=&(\lambda, 0, 0)+t(0, -1, -1).
	\end{aligned}
	\label{eq:10}
	\end{equation}
are contained in $S$.

	For $x=-\lambda$, the same is true for
		\begin{equation}
			\begin{aligned}
	l'_1(t)=&(-\lambda, 0, 0)+t(0, -1, -1),\\
	l'_2(t)=&(-\lambda, 0, 0)+t(0, 1, -1).
	\end{aligned}
	\label{eq:11}
	\end{equation}
	
	Observe that we can equally well get the lines from Equation (\ref{eq:9}) by a  rotation of $\pi$ 
radians of the lines obtained in Equation (\ref{eq:6}) around $z$--axis, 
which is to be expected since we are looking at diametrically opposite points in the cylinder. 
In fact, rotating $r_1(t)$ we achieve
	\[\left[\begin{matrix}
	\cos\pi&-\sin\pi&0\\
	\sin\pi&\cos\pi&0\\
	0&0&1
	\end{matrix}\right]\left[\begin{matrix}
	x_0-\frac{t\sqrt{\lambda^2-x_0^2}}{\lambda}\\
	y_0+\frac{tx_0}{\lambda}\\
	-t
	\end{matrix}\right]=\left[\begin{matrix}
	-x_0+\frac{t\sqrt{\lambda^2-x_0^2}}{\lambda}\\
	-y_0-\frac{tx_0}{\lambda}\\
	-t
	\end{matrix}\right],
		\]
		which is exactly $s_2(t)$. For $r_2(t)$ rotated by $\pi$ radians around of $z$--axis we get
		$$\left[\begin{matrix}
		\cos\pi&-\sin\pi&0\\
		\sin\pi&\cos\pi&0\\
		0&0&1
		\end{matrix}\right]\left[\begin{matrix}
		x_0+\frac{t\sqrt{\lambda^2-x_0^2}}{\lambda}\\
		y_0-\frac{tx_0}{\lambda}\\
		-t
		\end{matrix}\right]=\left[\begin{matrix}
		-x_0-\frac{t\sqrt{\lambda^2-x_0^2}}{\lambda}\\
		-y_0+\frac{tx_0}{\lambda}\\
		-t
		\end{matrix}\right],
		$$
		exactly $s_1(t)$. By the same argument, we see that 
$l'_1(t)$ is a rotation of $l_1(t)$ and  $l'_2(t)$, a rotation of $l_2(t)$.
	\\
	
	Let us define $F_1$ as the union of the lines obtained from (\ref{eq:6}) 
and (\ref{eq:9}) together with $l_1(t)$ and $l'_1(t)$. 
Similarly, let $F_2$ be the union of the families of the lines obtained from (\ref{eq:7}) 
together with (\ref{eq:8}), this time appending $l_2(t)$ and $l'_2(t)$. 
These families are disjoint since they come from different linearly independent systems of equations.
\end{proof}

\begin{figure}[h]
	\centering      
	\includegraphics[width=0.7\textwidth]{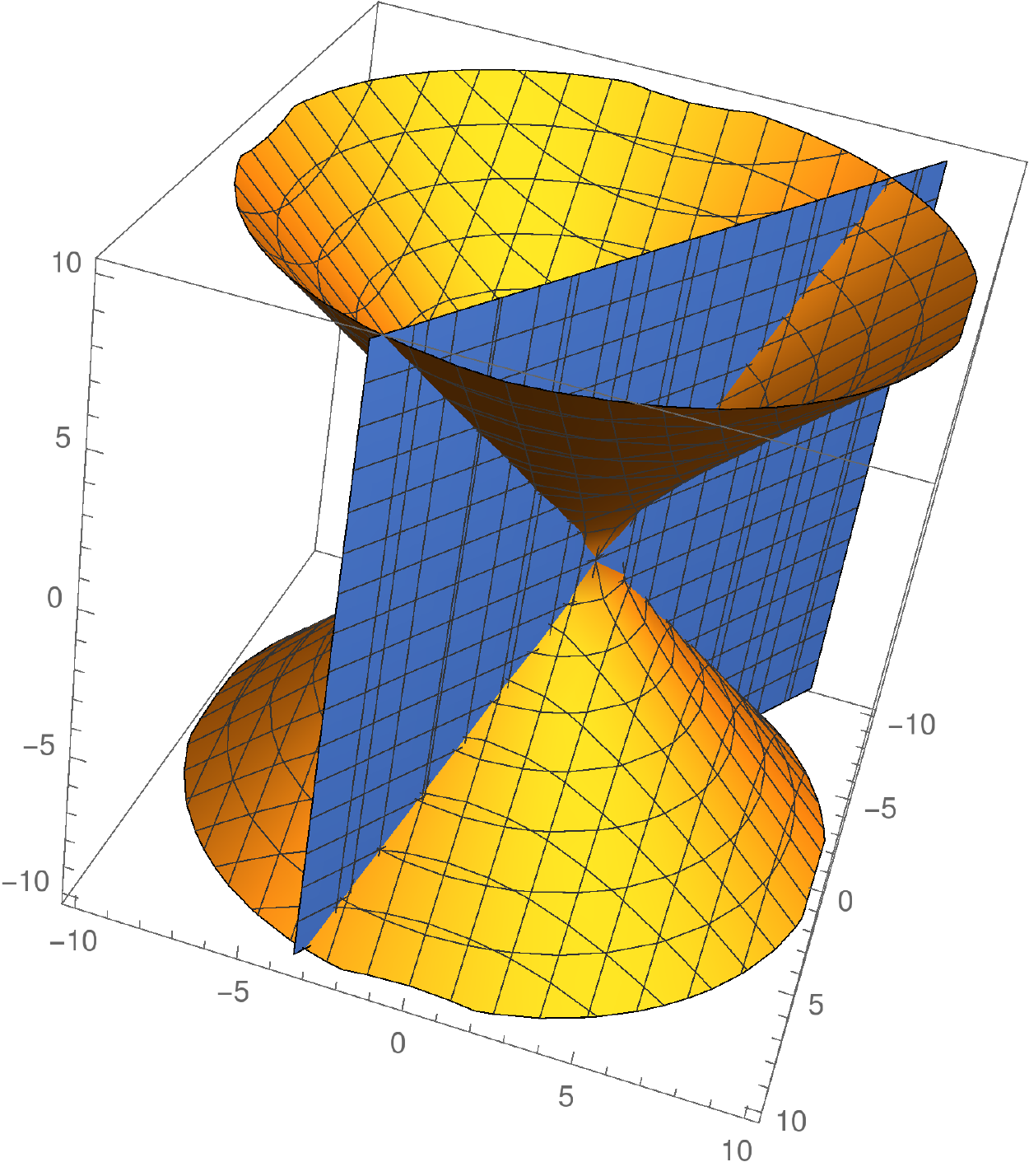}      
	\caption{The intersection of the plane in (\ref{eq:4}) with the one--sheeted hyperboloid.}
	\label{fig:intersection}
\end{figure}

\begin{proposition}
The families $F_i$, for $i=1,2$, satisfy the following properties.

	$\bullet$ \label{exist-rot} For any two lines $a, b\in F_i$, there exists a rotation $R_{\theta}$ 
around the $z$--axis such that $R_\theta a=b$.

	$\bullet$ If $a\in F_i$ and $b$ is such that there exists a rotation 
$R_\theta$ around the $z$--axis such that $R_\theta a=b$, then $b\in F_i$.
\end{proposition}
\begin{proof}
	We prove the proposition just for $F_1$, the case $F_2$ is completely analogous.

By Lemma \ref{lemma2}, if $a, b\in F_1$, then the lines pass through 
		a point $(x_0, y_0, 0)\in S$, so they have the shape
\begin{align*}
r_1(t)&=(x_0, y_0, 0)+t\left(-\frac{\sqrt{\lambda^2-x_0^2}}{\lambda},\frac{x_0}{\lambda}, -1\right),& y_0>0,\\
s_2(t)&=(x_0, y_0, 0)+t\left(\frac{\sqrt{\lambda^2-x_0^2}}{\lambda},\frac{x_0}{\lambda}, -1\right),& y_0<0.
\end{align*}
Note that it is enough to show that for each $a\in F_1$ there exists a rotation 
$R_\theta$ such that $R_\theta l_1=a$. This is so because $R_\phi l_1=b$ implies 
$R_\phi R_\theta^{-1}a=b.$ The case $a=l_2'$ was the content of Lemma \ref{lemma2}, 
so we suppose $a\neq l_2'$. Direct calculation yields then
\begin{equation}
\begin{bmatrix}
		\cos\theta & -\sin\theta & 0\\
		\sin\theta & \quad\cos\theta & 0\\
		0&0&1
		\end{bmatrix}\begin{bmatrix}
		\lambda\\
		t\\
		-t
		\end{bmatrix}=\begin{bmatrix}
		\lambda\cos\theta  -t\sin\theta \\
		\lambda\sin\theta +t\cos\theta\\
		-t
		\end{bmatrix}.
		\label{eq:12}
		\end{equation}
If $a$ passes through $(x, y, 0)\in S$, we choose $\theta$ so that 
$x=\lambda\cos\theta$ and $y=\lambda\sin\theta$ (this is possible given that for each $(x_0 , y_0, 0)\in S$ 
we have $x_0^2+y_0^2=\lambda^2$). By using $y=\sqrt{\lambda^2-x^2}$ whenever $y>0$, we get
$$\begin{bmatrix}
x+\frac{t\sqrt{\lambda^2-x^2}}{\lambda}\\
y+\frac{tx_0}{\lambda}\\
-t
\end{bmatrix};$$
while for $y<0$, we use $y=-\sqrt{\lambda^2-x^2}$ and obtain
$$\begin{bmatrix}
x-\frac{t\sqrt{\lambda^2-x^2}}{\lambda}\\
y+\frac{tx_0}{\lambda}\\
-t
\end{bmatrix}.$$

To verify the second statement it is enough to show that if there exists a rotation 
$R_\theta$ such that $R_\theta l_1=b$, then $b\in F_1$, since by Item \ref{exist-rot} 
there exists $R_\phi$ such that $R_\phi l_1=a$. 
If this is so, there exists $\alpha$ such that $R_\alpha a=b$ if and only if there exists 
$\phi$ such that $R_\phi R_\alpha l_1=b$.
By changing variables on Equation (\ref{eq:12}), with $x_0=\lambda\cos\theta$ and $y_0=\lambda\sin\theta$, we reach
$$\begin{bmatrix}
x_0\pm\frac{t\sqrt{\lambda^2-x^2}}{\lambda}\\
y_0+\frac{tx_0}{\lambda}\\
-t
\end{bmatrix};$$
which is exactly the expression of the lines given by the planes (\ref{eq:6}) and  (\ref{eq:9}).  
Finally, by definition of $F_1$, we have $R_\theta l_1\in F_1$.
\end{proof}
\begin{proposition}
The one--sheeted hyperboloid $S$ is a doubly ruled surface.
\end{proposition}
\begin{proof}
	Fix $(x_0, y_0, z_0)\in S$. We look again at the line $l_1\in F_1$, where $l_1(t)=(\lambda, 0, 0)+t(0, 1, -1)$. 
We will show that there exists a rotation $R_\theta$ around $z$--axis such that 
$R_\theta l_1(t)=(x_0, y_0, z_0)$ for some $t\in\R$ 
(observe that by the second item in the above proposition, we already have $R_\theta l_1\in F_1$). 
By Equation (\ref{eq:12}), we get
	\[R_\theta l_1(t)=(\lambda\cos\theta-t\sin\theta, \lambda\sin\theta+t\cos\theta, -t),\]
	and by letting $t=-z_0$, we obtain
	\begin{equation}
	R_\theta l_1(-z_0)=(\lambda\cos\theta+z_0\sin\theta, \lambda\sin\theta-z_0\cos\theta, z_0).\label{eq:13}\end{equation}
	Varying $\theta$ in Equation (\ref{eq:13}) ables us to trace the entire level curve $S$ at $z=z_0$, 
which  is a circle of radius $\lambda^2-z_0^2$. 
Therefore, since $(x_0, y_0, z_0)\in S$ holds, there exist $\theta$ subject to 
$\lambda\cos\theta+z_0\sin\theta=x_0$ and $\lambda\sin\theta-z_0\cos\theta=y_0$. 
Thus  $R_\theta l_1$ is a line in $F_1$ subject to $(x_0, y_0, z_0)\in R_\theta l_1$.
In the same way, it is not hard to show that there exists a rotation 
$R_\phi$ such that $R_\phi l_2$ contains the point $(x_0, y_0, z_0)$. 
We conclude that $r(\theta, t)=R_{\theta}l_{1}(t)$ and $s(\theta, t)=R_{\theta}l_{2}(t)$ 
are parametric equations for $S$. Hence, $S$ is ruled by both $F_1$ and $F_2$.
\end{proof}

\section{The tangent spaces to $\mathcal{O}(\lambda A)$}
\label{section5}
The adjoint orbit $\mathcal{O}(\lambda A)$ is a surface in $\R^3$. 
The goal of this section is to depict the tangent space to  $\mathcal{O}(\lambda A)$ 
and determine its relation with the image of the adjoint representation of the Lie algebra $\mathfrak{sl}(2,\R)$.
Our starting point is the following proposition, 
which provides an identification of said tangent space.
\begin{proposition} Let $\ad$ be the adjoint representation of $\mathfrak{g}$. For $H\in\mathcal{O}(\lambda A)$  
we have
	$$\mathrm{Im}(\ad(H))= T_H\mathcal{O}(\lambda A).$$
\end{proposition}
\begin{proof}
Notice that every curve passing through $H$ in $\mathcal{O}(\lambda A)$ has the form $g_tHg^{-1}_t$, 
where $g\colon[-\epsilon,\epsilon]\rightarrow \mathrm{SL}(2, \R)$ 
smoothly satisfies $g_0=\mathrm{Id}$. 
Thus, every tangent vector $v\in T_H\mathcal{O}(\lambda A)$ can be written as
$$
v=\frac{d}{dt}g_tHg^{-1}_t\bigg|_{t=0},
$$
for some smooth curve $g_t$. However, if $\frac{d}{dt}g_t\bigg|_{t=0}=X\in \mathfrak{sl}(2, \R)$, then
\begin{equation}
\frac{d}{dt}g_tHg^{-1}_t\bigg|_{t=0}=\frac{d}{dt}\mathrm{Ad}_{g_t}H\bigg|_{t=0}=d(\mathrm{Ad}_{\mathrm{Id}})(X)H=\mathrm{ad}(X)H=[X,H].
\label{eq:17}
\end{equation}
Hence we get $v\in \mathrm{Im}(\ad(H))$, which forces the inclusion $T_H\mathcal{O}(\lambda A)\subset\mathrm{Im}(\ad(H))$.
 
Conversely, given $X_0=[X, H]$, just take some smooth curve $g_t$ subject to 
\begin{eqnarray*}
\frac{d}{dt}g_t\bigg|_{t=0}=X_0 & \textup{ with, } & g_0=\mathrm{Id}\in \mathrm{SL}(2, \R),
\end{eqnarray*}
and plug it into Equation (\ref{eq:17}): we obtain the desired result.
\end{proof}
Let $A, B, C$ be the basis of $\mathfrak{sl}(2,\mathbb{R})$ given in (\ref{baseSL2}). 
Since they satisfy the relations $
[B, A]=C, [C, A]=B$, and $[B, C]=A$, for any $H\in\mathfrak{g}$ written as $H=xA+yB+zC$  we get
\[\ad(H)=\begin{bmatrix}
0&-z&-y\\
z&0&x\\
y&-x&0
\end{bmatrix}.\]
Since  $\dim\ker\ad(H)=1$, by taking any two column vectors in $\ad(H)$ we have
\[T_H\mathcal{O}(\lambda A)= \spn \{zB+yC,xB-yA\},\]
whenever $H\in\mathcal{O}(\lambda A)$. As $H=xA+yB+zC$ implies that 
$x, y, z$ satisfy $x^2+y^2-z^2=\lambda^2$, we get
\[T_H\mathcal{O} (\lambda A)=\spn\{zB+yC,xB-yA\colon x^2+y^2-z^2=\lambda^2\}.\]

\section{Morse theory on the adjoint orbit $\mathcal{O}(A)$}
We use the adjoint orbit $\mathcal{O}(\lambda A)$ 
to construct an example which shows how the compactness hypothesis is essential to the  Morse-Smale theorem. 
We take $\lambda=1$ to ease computations. 

Let $M$ be a manifold and $f\colon M\rightarrow\R$ a smooth function. 
A critical point $p$ of $f$ is \textbf{non-degenerate} if the Hessian matrix of $f$ in $p$ is non-degenerate. 
If all critical points of $f$ are non-degenerate, we say $f$ is a \textbf{Morse function}.

\begin{theorem}[Morse--Smale]\cite[Lem.\thinspace2.23]{LN}.
	Let $M$ be a compact manifold without boundary and $f\colon M\rightarrow\R$ a  Morse function. 
If $\phi_p(t)$ is the trajectory of the gradient vector field  $\nabla f$ at $p$, then both limits 
	\[\lim_{t\rightarrow \infty}\phi_p(t)\qquad\textup{ and }\qquad\lim_{t\rightarrow -\infty}\phi_p(t)\]
	exist, in fact, they are critical points of $f$.
	\label{thm:4}$\hspace{4cm}\square$
\end{theorem}
Dynamics of the gradient vector field
We study the behaviour of a gradient vector field restricted to  $\mathcal{O}(A)$.
For that, we consider the function $f(x, y, z)=yz$. 
The gradient of $f$ (with respect to the canonical inner product) is
\[\nabla f(x, y, z)=\frac{\partial f}{\partial x}(x, y, z)e_1+\frac{\partial f}{\partial y}(x, y, z)e_2+\frac{\partial f}{\partial z}(x, y, z)e_3=ze_2+ye_3.\]
Therefore, the gradient matrix of $f$ in the canonical basis is
\[\nabla f=\left(\begin{array}{c}
0\\z\\y
\end{array}\right).\]
\begin{proposition}
	The gradient vector field  $\nabla f$ is tangent to $\mathcal{O}(A)$.
\end{proposition}
\begin{proof}
	Let us consider the relation $g(x,y)=\pm\sqrt{x^2+y^2-1}$  which describes the one--sheeted hyperboloid. 
We look first at the case $z_0\geq0$ (hence $g(x,y)=\sqrt{x^2+y^2-1}$). 
The normal vector to the surface at a given point $(x_0, y_0, z_0)$ is
	\[\vec{n}=\left(\frac{x_0}{\sqrt{x_0^2+y_0^2-1}}, \frac{y_0}{\sqrt{x_0^2+y_0^2-1}}, -1\right).\]
	Thus, taking the inner product between $\vec{n}$ and $\nabla f(x_0, y_0, z_0)=(0, z_0, y_0)^T$, yields
	\[\langle \vec{n}, \nabla f(x_0, y_0, z_0) \rangle=\frac{y_0z_0}{\sqrt{x_0^2+y_0^2-1}}-y_0;\]
	and after using $z_0=\sqrt{x_0^2+y_0^2-1}$ we reach
	\[\langle \vec{n}, \nabla f(x_0, y_0, z_0) \rangle=0.\]
	Hence, $\nabla f$ takes the point $(x_0, y_0, z_0)$ to a vector tangent to the surface, 
and so, $\nabla f$ is tangent to $\mathcal{O}(A)$ for $z_0\geq0$.\\
	
	Similarly for $z_0\leq0$, we use $g(x,y)=-\sqrt{x^2+y^2-1}$ instead and get
	\[\vec{n}=\left(-\frac{x_0}{\sqrt{x_0^2+y_0^2-1}},- \frac{y_0}{\sqrt{x_0^2+y_0^2-1}}, -1\right).\]
	With $z_0=-\sqrt{x_0^2+y_0^2-1}$, we obtain
	\[\langle \vec{n}, \nabla f(x_0, y_0, z_0) \rangle=-\frac{y_0z_0}{\sqrt{x_0^2+y_0^2-1}}-y_0=0.\]

	In either case, we conclude that the gradient vector field $\nabla f$ is tangent to $\mathcal{O}(A)$. 
	\end{proof}
\begin{figure}[h]
	\centering      
	\includegraphics[width=0.8\textwidth]{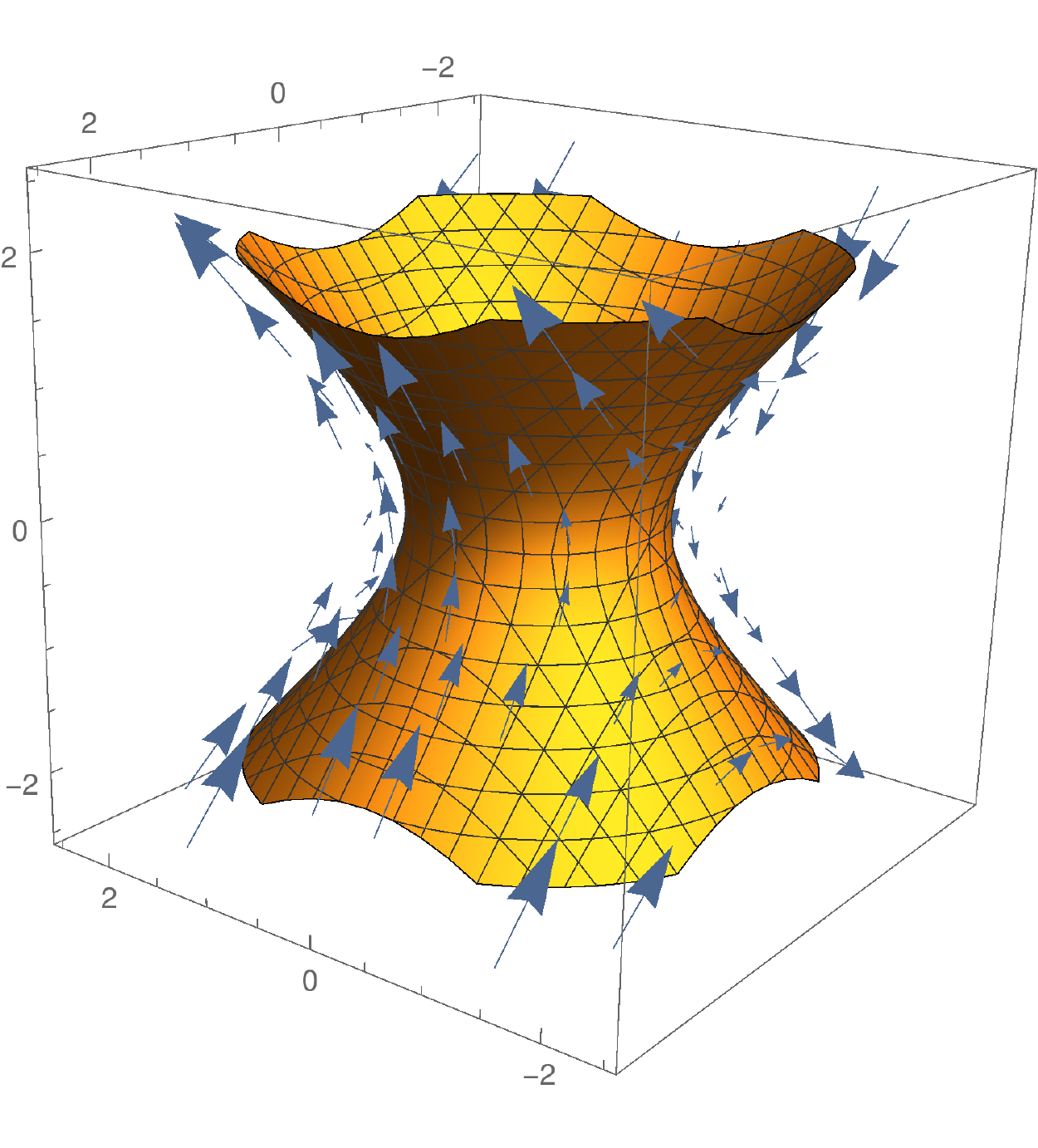}      
	\caption{The vector field restricted to $\mathcal{O}(A)$.}
	\label{fig:vectorfieldrestricted}
\end{figure}
\begin{proposition}
	The function $f(x, y, z)=yz$ restricted to $\mathcal{O}(A)$ is a Morse function.
\end{proposition}
\begin{proof}
Notice that $(1, 0, 0)$ and $(-1, 0, 0)$ 
are the singularities of the restriction gradient vector field to the orbit $\mathcal{O}(A)$. 
Let $\mathrm{Hess}(f)$ be the Hessian matrix of $f$, namely,
\[\mathrm{Hess}(f)=\begin{bmatrix}
0&0&0\\
0&0&1\\
0&1&0
\end{bmatrix}.\]
Note that the restriction of $\mathrm{Hess}(f)$ to each of the tangent spaces at 
$(1, 0, 0)$ and $(-1, 0, 0)$ is non-degenerate. 
In fact, using the results of Section \ref{section5} and  identifying $(1, 0, 0)$ and 
$(-1, 0, 0)$ with the matrices $A$ and $-A$ in $\mathcal{O}(A)$, respectively,  
we get $T_A\mathcal{O}(A)=\spn\{(0, 0, 1),(0, 1, 0)\}$ and $T_{-A}\mathcal{O}(A)=\spn\{(0, 0, -1),(0, -1, 0)\}$. 
It is not hard to conclude the equalities
\begin{align*}
T_A\mathcal{O}(A)\cap \ker\mathrm{Hess}(f)&=0,\\
T_{-A}\mathcal{O}(A)\cap \ker\mathrm{Hess}(f)&=0,
\end{align*}
which implies that $f|_{\mathcal{O}(A)}$ is a Morse function.
\end{proof}
Using the dynamics above, we  obtain trajectories of $\nabla f$ which 
are not complete, thus  showing that the hypothesis of compactness in  Theorem \ref{thm:4}
is fundamental.\\

 The trajectories of the gradient $\nabla f$ restricted to $T_A\mathcal{O}(A)$
 are solutions of
 the following linear system of differential equations
\[\begin{bmatrix}
y'(t)\\z'(t)
\end{bmatrix}=\begin{bmatrix}
0&1\\1&0
\end{bmatrix}\begin{bmatrix}
y(t)\\z(t)
\end{bmatrix}.\]
Since $1$ and $1$ are eigenvalues of the linear part, 
with eigenvectors $v_1=(-1, 1)$ and $v_2=(1, 1)$, respectively, it follows that the general solution has the form
\[\begin{bmatrix}
y(t)\\z(t)
\end{bmatrix}=c_1e^{-t}v_1+c_2e^tv_2.\]

Setting $c_1=0$ and $c_2=1$, we obtain $\gamma_1(t)=e^t(1, 1)$. 
Note that $\lim_{t\rightarrow -\infty}\gamma_1(t)=(0, 0)$ but the limit 
when $t\rightarrow\infty$ does not make sense in $\mathcal{O}(A)$.\\

On the other hand, taking $c_2=0$ and $c_1=1$, we get $\gamma_2(t)=e^{-t}(-1, 1)$. 
Where $\lim_{t\rightarrow \infty}\gamma_2(t)=(0, 0)$ but the limit when $t\rightarrow-\infty$ does not exist.\\

Considering $\gamma_1$ and $\gamma_2$ in the tangent space 
$T_A\mathcal{O}(A)$, we have that these are the lines $(1, t, t)$ and $(1, -t, t)$. 
Moreover, they correspond to the lines $l_1$ and $l_2$ as described in 
Equation (\ref{eq:10}) in Lemma \ref{lemma2} of Section \ref{section3.1}. 
So they are contained in $\mathcal{O}(A)$.
Summarizing, we obtain two trajectories 
$\gamma_1(t)$ and $\gamma_2(t)$ of the gradient $\nabla f$ 
whose limit points do not belong to the orbit, namely, do not satisfy 
the conclusions of 
Theorem \ref{thm:4}. 
This  happens because the one--sheeted hyperboloid $\mathcal{O}(A)$ is a non-compact submanifold of $\R^3$.
\begin{figure}[h]
	\centering      
	\includegraphics[width=0.45\textwidth]{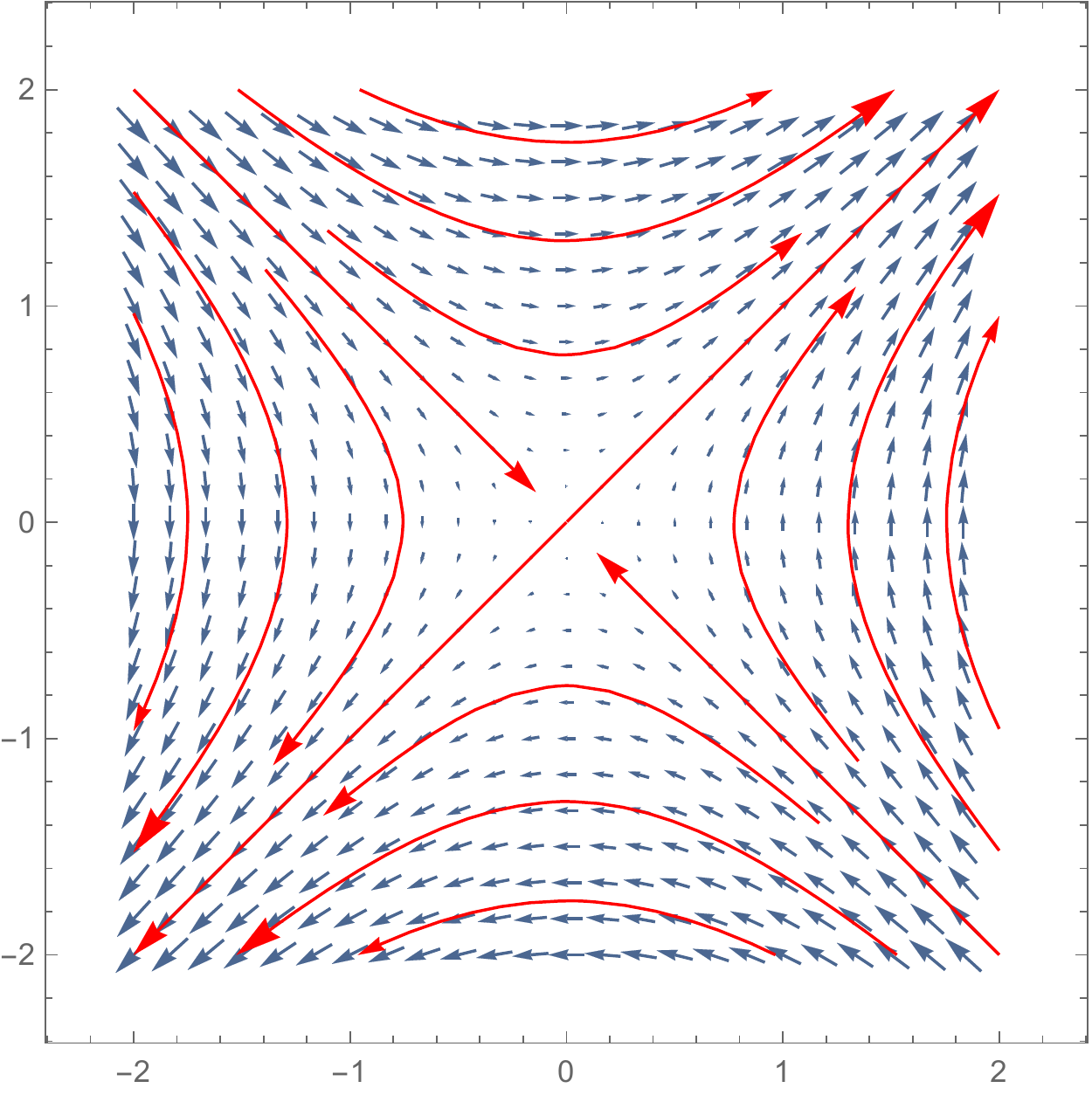}      
	\caption{The gradient $\nabla f$ restricted on $T_A\mathcal{O}(A)$ at $A=(1, 0, 0)$.}
	\label{fig:vectorfieldrestrictedonA}
\end{figure}

\section{A symplectic structure in $\mathcal{O}(A)$}
Here we realize the adjoint orbit $\mathcal{O}(A)$ as a symplectic manifold. 
We follow the construction by Kirillov--Kostant--Souriau \cite{Kir1,Kir2}. 
First we construct the symplectic form on the coadjoint orbit and then, 
using the Killing form, we induce the symplectic structure on the adjoint orbit  $\mathcal{O}(A)$. 
For a more general study of symplectic geometry on adjoint orbits we refer the reader to \cite{ABB} and \cite{GGSM1}.

In order to perform the construction, we start by recalling some basic definitions of symplectic geometry.

	Let $V$ be a real vector space and $\omega\colon V\times V\rightarrow \R$ a 
skew--symmetric bilinear form. We say that $\omega$ is a \textbf{symplectic form} if it is non-degenerate, that is, 
	$\omega(u,v)=0$ for all $v$ implies $u=0.$
	In this case, we say that $(V, \omega)$ is a \textbf{symplectic vector space}.

	Let $M$ be a manifold. We say that a $2$--form $\omega\in\Omega^2(M)$ is \textbf{non-degenerate} 
if the $1$--form $\omega_x=\omega(x, \cdot)$ is non-degenerate for each $x\in M$. 
Thus, for every $x\in M$, the tangent space $T_xM$ is a symplectic vector space.

	A \textbf{symplectic structure} on $M$ is a $2$--form $\omega\in\Omega^2(M)$ 
which is non-degenerate and closed. In this case, we say that $(M, \omega)$ is a \textbf{symplectic manifold}.

Now we define the coadjoint representation, 
which is the dual of the adjoint representation and will allow us to define coadjoint orbits.
First, let us consider the natural pairing between $\mathfrak{g}^*$ and $\mathfrak{g}$ given by
 
 \[	\begin{aligned}
\langle\,\, , \,\rangle\colon \mathfrak{g}^{\ast}\times\mathfrak{g}&\rightarrow \R\\
(\xi, X)&\mapsto\langle\xi,X\rangle=\xi(X).
 \end{aligned}\]
 For $\xi\in\mathfrak{g}^*$, we define $\Ad_g^{*}\xi$ by the rule
 $$\langle\Ad_g^*\xi,X\rangle=\langle\xi,\Ad_{g^{-1}}X\rangle,\qquad X\in\mathfrak{g}.$$

The \textbf{coadjoint representation of $G$ on} $\mathfrak{g}^*$ is by definition
		\[\begin{aligned}
		\Ad^*\colon G&\rightarrow \Aut(\mathfrak{g}^*)\\
		g&\mapsto\Ad^*_g.
		\end{aligned}\]

Similarly, we have a coadjoint representation of $\mathfrak{g}$ on $\mathfrak{g}^*$ given by
\[\begin{aligned}
\ad^*\colon \mathfrak{g}&\rightarrow \mathfrak{gl}(\mathfrak{g}^*)\\
u&\mapsto\ad^*_u.
\end{aligned}\]
To be more explicit, given $u\in\mathfrak{g}$ and $\xi\in\mathfrak{g}^*$, 
we have $\langle \ad^*_u(\xi),v\rangle =-\xi([u,v])$. Here $[\,\, , \,]$ is the Lie bracket on $\mathfrak{g}$.

Let us consider $\xi\in\mathfrak{g}^*$ and denote by 
\[\mathcal{O}^*=\{\varphi\in\mathfrak{g}^*\colon\textup{  there is } u\in G, \Ad_u^*(\xi)=\varphi\}\] 
the \textbf{coadjoint orbit} of $\xi$. Since the vectors $\Ad_u^*(\xi)$ span the tangent space $T_\xi\mathcal{O}^*$
we have
\[T_\xi\mathcal{O}^*=\{\ad_{u}^*(\xi)\colon u\in\mathfrak{g}\}.\]
Note that for a fixed $\xi\in\mathfrak{g}^*$, the value of $\xi[u, v]$ 
depends just on $\ad_{u}^*$ and $\ad_{v}^*$ at the point $\xi$. 
In fact, if $\ad_{u}^*(\xi)=\ad_{u'}^*(\xi)$, then 
\[\xi(u-u', v)=(\ad_u^*-\ad_{u'}^*)(\xi)(v)=0,\]
for each $v\in\mathfrak{g}$. 
Thus, the following definition of a skew-symmetric bilinear form on $T_\xi\mathcal{O}^*$ makes sense.

For $\xi\in\mathfrak{g}^*$ fixed, we define a skew-symmetric bilinear form on $T_\xi\mathcal{O}^*$ by $$\omega_{\xi}(\ad_{u}^*(\xi), \ad_{v}^*(\xi))=\xi([u, v]).$$

\begin{lemma}\label{lemma3}
	For each $\xi\in\mathfrak{g}^*$ the form $\omega_{\xi}$ is non-degenerate.
\end{lemma}
\begin{proof}
	Note that if \[\omega_{\xi}(\ad_{u}^*(\xi), \ad_{v}^*(\xi))=0,\] 
	for all $v\in\mathfrak{g}$, then $\xi([u, v])=0=-\xi([u, v])$ and therefore $\ad_{u}^*(\xi)=0$.
\end{proof}
Since we have $[\Ad_g(u),\Ad_g(v)]=\Ad_g([u, v])$, we get the equality
$$
\Ad_g^*\xi([\Ad_g(u),\Ad_g(v)])=\Ad_g^*\xi(\Ad_g([u, v]))=\xi([u, v]).
$$
Thus, $\omega_{\xi}$ defines a point-wise form $\omega$ on $\mathcal{O}$.
\begin{lemma}\label{lemma4}
The $2$--form $\omega$ is closed, that is, satisfies $\mathrm{d}\omega=0$.
\end{lemma}
\begin{proof}
		We analyse $\omega$ point-wise. 
For any $\xi\in\mathfrak{g}^*$, given $x, y$ and $z$ in $\mathfrak{g}$ set 
$X=\ad_{x}^*(\xi)$, $Y=\ad_{y}^*(\xi)$ and $Z=\ad_{z}^*(\xi)$. We have then
	\begin{align*}
	\mathrm{d}\omega_\xi(X, Y, Z)=&\frac{1}{3}\left(X\omega_{\xi}(Y, Z)-Y\omega_{\xi}(X, Z)+Z\omega_{\xi}(X, Y)\right)+\\
	&+ \frac{1}{3}(-\omega_{\xi}([X, Y], Z)+\omega_{\xi}([X, Z], Y)-\omega_{\xi}([Y, Z],X)).
	\end{align*}
	Note that all the directional derivatives vanish, since $\omega_{\xi}$ is constant relative to $\xi$. 
Thus using Jacobi identity, we reach
	\begin{align*}
	\mathrm{d}\omega_\xi(X, Y, Z) &=\frac{1}{3}(-\omega_{\xi}([X, Y], Z)+\omega_{\xi}([X, Z], Y)-\omega_{\xi}([Y, Z], X)),\\
	&=\frac{1}{3}(-{\xi}([[x, y], z])+{\xi}([[x, z], y])-{\xi}([[y, z], x])),\\
	&=\frac{1}{3}{\xi}(-[[x, y], z])+[[x, z], y]-[[y, z], x],\\
	&=\frac{1}{3}{\xi}(0),\\
	&=0.
	\end{align*}
	Therefore $\omega$ is closed.
\end{proof}
\begin{theorem}
	Let $\mathcal{O}^*\subset\mathfrak{g}^*$ be a coadjoint orbit. 
Then $\omega_{\xi}$ defines a symplectic structure on $\mathcal{O}^*$.
\end{theorem}
\begin{proof}
	This is a direct consequence of Lemmas \ref{lemma3} y \ref{lemma4}.
\end{proof}
\begin{remark}
	This symplectic structure on the coadjoint orbit is canonical and is called the \textbf{Kirillov--Kostant--Souriau form}.
\end{remark}
Now we show below how to endow the adjoint orbit with a symplectic structure which comes from the symplectic form constructed above.
\begin{proposition}
	Suppose $\mathfrak{g}$ admits an $\Ad$--invariant inner product, that is, one subject to
	\[\langle\Ad_g(u),\Ad_g(v)\rangle=\langle u, v\rangle,\]for each $g\in G$. 
Then the identification $\mathfrak{g}\cong\mathfrak{g}^*$ induced by this inner 
product also provides an isomorphism between the adjoint and coadjoint representations.
\end{proposition}
\begin{proof}
	The isomorphism of vector spaces $\mathfrak{g}\cong\mathfrak{g}^*$ is given by
	\begin{equation}
	\begin{array}{rrcl}
	&\varphi\colon\mathfrak{g}&\rightarrow &\mathfrak{g}^*\\
	&v&\mapsto&I_v,
	\end{array}
	\label{eq:15}
	\end{equation}
	where $I_v(u)=\langle u, v\rangle$, for each $u\in\mathfrak{g}^*$.
	We want to show that $\varphi$ is an  isomorphism of representations as well, 
namely, an isomorphism of Lie algebras for which the following diagram
\[
\xymatrix{
	\mathfrak{g}\ar[r]^\varphi \ar[d]_{\Ad_g} & \mathfrak{g}^* \ar[d]^{\Ad_g^*} \\
	\mathfrak{g}\ar[r]_\varphi & \mathfrak{g}^* 
}
\]
is commutative.

Since $\varphi$ is an isomorphism of vector spaces, whenever we get $\mathfrak{g}\simeq V$ as vector spaces endowed with Lie brackets, 
then we can make $V^*$ into a Lie algebra by defining a Lie bracket as
\[[a, b]_{*}=\varphi([\varphi^{-1}(a), \varphi^{-1}(b)]), \textup{ for each } a, b\in V^*.\]
Thus, we get $\mathfrak{g}^*\cong(V^*, [\cdot, \cdot]_{*})$ and $\varphi$ is a Lie algebra homomorphism. 
In fact, it is easy to check the equality
\[\varphi([a, b])=[\varphi(a), \varphi(b)]_{*}.\]
Next, let $v\in\mathfrak{g}$ be a fixed element. Since $\Ad_g$ is invertible, 
there exists $w\in\mathfrak{g}$ such that $\Ad_g(w)=v$. 
Using $\Ad$--invariance for the inner product, we obtain
\[\varphi(\Ad_g(u))(v)=\langle\Ad_g(u), v\rangle= \langle\Ad_g(u), \Ad_g(w)\rangle= \langle u, w\rangle. \]
In the same way, we get
\[\Ad_g^*(\varphi(u))(v)=\varphi(u)(\Ad_{g^{-1}}(v))=\langle u, \Ad_{g^{-1}}(v)\rangle=\langle u, w\rangle.\]
Therefore, the diagram is commutative and $\varphi$ is a representation isomorphism.
\end{proof}

\begin{remark}
Notice that the above result holds in a more general context where the product is only non-degenerate and not necessarily positive definite, and hence not a inner product. It follows from the fact that the linear map induced by a non-degenerate product between $\mathfrak{g}$ and $\mathfrak{g}^*$ is an isomophism. The proof of the above proposition also holds in this case, since it only requires the existence of such isomorphism.
\end{remark}

	Let $\mathfrak{g}$ be a Lie algebra over a field $\mathbb{F}$. 
The \textbf{Killing form} on $\mathfrak{g}$ is the map
	\[\begin{array}{rlll}
	B\colon\mathfrak{g}\times\mathfrak{g}&\rightarrow &\mathbb{F}\\
	(x, y)&\mapsto&B(x, y)=\tr (\ad_x\circ\ad_y).
	\end{array}\]

\begin{proposition}
	The Killing form is $\Ad$-invariant.
\end{proposition}
\begin{proof}
	In fact, we have
	\begin{align*}
	B(\Ad_g(x), \Ad_g(y))&=\tr (g\,\circ\,\ad_x\circ\ad_y\circ \,g^{-1})\\
	&=\tr (\ad_x\circ\ad_y )\\
	&=B(x,y).
	\end{align*}
\end{proof}
\begin{proposition}
	The Killing form is  symmetric and bilinear.
\label{prop:7}\end{proposition}
\begin{proof}
Symmetry follows from the property $\tr(MN)=\tr(NM)$. 
Thus, we have $$B(x, y)=\tr(\ad_x\circ \ad_y)=\tr(\ad_y\circ\ad_x)=B(y,x).$$
Since $\ad$ and trace are linear, we get
\begin{align*}
B(\alpha x+\beta y, z)&=\tr (\ad(\alpha x+\beta y)\circ \ad z)\\
&=\tr ((\alpha\ad x+\beta\ad y)\circ \ad z)\\
&=\alpha\tr (\alpha\ad x+\ad z)+\beta\tr (\ad y\circ\ad z)\\
&=\alpha B(x,z)+\beta B(y, z),
\end{align*}	
for each $x, y, z\in\mathfrak{g}$. Thus, $B$ is linear on the first entry. 
By the symmetry we get the same for the second entry.
\end{proof}
Hanceforth, for simplicity we use the notation $\langle a, b\rangle = B(a, b)$.

\begin{proposition} The Killing form is non-degenerate. \end{proposition}
\begin{proof}

Since $[B, A]=C, [C, A]=B$, and $[B, C]=A$ hold, we get 
\[\ad(A)=\begin{bmatrix}
0&0&0\\
0&0&-1\\
0&-1&0\\
\end{bmatrix},\,\, \ad(B)=\begin{bmatrix}
0&0&1\\
0&0&0\\
1&0&0\\
\end{bmatrix},\,\, \ad(C)=\begin{bmatrix}
0&-1&0\\
1&0&0\\
0&0&0\\
\end{bmatrix}.\]
Direct computation yields
\begin{align*}
\langle A, A\rangle&=\tr(\ad (A)\circ \ad(A))=\tr \left(\begin{bmatrix}
0&0&0\\
0&1&0\\
0&0&1\\
\end{bmatrix}\right)=2,\\
\langle B, B\rangle&=\tr(\ad (B)\circ \ad(B))=\tr \left(\begin{bmatrix}
		1&0&0\\
		0&0&0\\
		0&0&1\\
	\end{bmatrix}\right)=2,\\
\langle C, C\rangle&=\tr(\ad (C)\circ \ad(C))=\tr \left(\begin{bmatrix}
	-1&0&0\\
	0&-1&0\\
	0&0&0\\
\end{bmatrix}\right)=-2,\\
\langle B, A\rangle&=\tr(\ad (B)\circ \ad(A))=\tr \left(\begin{bmatrix}
0&-1&0\\
0&0&0\\
0&0&0\\
\end{bmatrix}\right)=0,\\
\langle A, C\rangle&=\tr(\ad (A)\circ \ad(C))=\tr \left(\begin{bmatrix}
0&0&0\\
0&0&0\\
-1&0&0\\
\end{bmatrix}\right)=0,\\
\langle B, C\rangle&=\tr(\ad (B)\circ \ad(C))=\tr \left(\begin{bmatrix}
0&0&0\\
0&0&0\\
0&-1&0\\
\end{bmatrix}\right)=0.
\end{align*}

Hence, if $H=xA+yB+zC$, expressing the operator $\langle H, \cdot \rangle$ on its matrix representation form we obtain

$$
\langle H, \cdot \rangle = 
\begin{bmatrix}
2x&2y&-2z\\
\end{bmatrix}.
$$

Therefore, the map $\langle H, \cdot \rangle$ is zero if and only if $H$ is zero.
\end{proof}

In this way, we identify the adjoint orbit $\mathcal{O}(A)$ with the coadjoint orbit 
$\mathcal{O}^*(A)=\{f\in\mathfrak{g}^*\colon\exists u\in G, \Ad_u^*(\varphi(A))=f\}$, 
where $\varphi$ is the map in (\ref{eq:15}). It follows  that we can induce on 
$\mathcal{O}(A)$ the symplectic structure built on $\mathcal{O}^*(A)$ as
\[\omega_p'(\ad_p(a), \ad_p(b)):=\omega_{\varphi(p)}(\ad_{\varphi(p)}^*(a), \ad_{\varphi(p)}^*(b))=\langle p, [a, b] \rangle,\]
for each $a, b, p\in\mathcal{O}(A)$.
\begin{corollary}
	The pair $(\mathcal{O}(A), \omega')$ is a symplectic manifold.
\hfill $\square$
\end{corollary}
\begin{remark}
	The Killing form is non-degenerate because we are working with a semisimple Lie algebra. 
This is essential to achieve the identification between adjoint and coadjoint orbits and, 
consequently, to perform the above construction.
\end{remark}
 In \cite{GGSM1}, the authors construct  another symplectic form which does not come from the Kirillov--Kostant--Souriau symplectic form. 
Their method involves Lie theory and the construction is performed directly on adjoint orbits of semisimple Lie groups.
It remains to carry out a complete classification of adjoint orbits in higher dimension.

\vspace{1cm}

%
%

\noindent
{\small Francisco Rubilar,\\
Universidad Cat\'olica del Norte,\\
Av. Angamos 0610, Antofagasta, Chile.\\
e-mail: francisco.rubilar@alumnos.ucn.cl}
\vspace{.5cm}

\noindent
{\small Leonardo Schultz \\
Universidade Estadual de Campinas\\
Cidade Universit\'aria Zeferino Vaz - Bar\~{a}o Geraldo\\
Campinas - SP, 13083-970, Brasil.\\
e-mail: leos.araujo@hotmail.com}

\end{document}